\documentclass[english]{article}
\usepackage[T1]{fontenc}
\usepackage[utf8]{inputenc}
\usepackage{lmodern}
\usepackage[a4paper]{geometry}
\usepackage{babel}

\usepackage{mathtools}
\usepackage{stmaryrd}
\usepackage{amsmath, amssymb, amsthm}

\usepackage{nicematrix} 

\usepackage{csquotes}
\usepackage[noblocks]{authblk}

\usepackage[dvipsnames]{xcolor}
\usepackage{soul}
\newcommand{\com}[1]{\textcolor{red}{[\hl{#1}]}}
\renewcommand{\com}[1]{}

\usepackage{tikz-cd}
\usepackage{biblatex}
\usepackage{enumitem}

\newtheorem{theorem}{Theorem}[section]

\newtheorem{proposition}[theorem]{Proposition}
\newtheorem{corollary}[theorem]{Corollary}
\theoremstyle{definition}
\newtheorem{definition}[theorem]{Definition}

\newtheorem{lemma}[theorem]{Lemma}
\theoremstyle{remark}
\newtheorem*{remark}{Remark}

\newcommand{\setR}{\mathbb{R}}
\newcommand{\setZ}{\mathbb{Z}}
\newcommand{\setC}{\mathbb{C}}
\newcommand{\setK}{\mathbb{K}}
\newcommand{\setH}{\mathbb{H}}
\newcommand{\setN}{\mathbb{N}}

\DeclareMathOperator{\ad}{ad}
\DeclareMathOperator{\Ad}{Ad}

\DeclareMathOperator{\End}{End}

\DeclareMathOperator{\GL}{GL}
\DeclareMathOperator{\SL}{SL}

\DeclareMathOperator{\Hom}{Hom}

\newcommand{\hMod}{\text{-}\mathrm{Mod}}

\DeclareMathOperator{\Vect}{Vect}
\newcommand{\hVect}{\text{-}\mathrm{Vect}}

\DeclareMathOperator{\id}{id}
\DeclareMathOperator{\Id}{Id}

\DeclareMathOperator{\SO}{SO}

\renewcommand{\O}{\operatorname{O}}
\newcommand{\SOp}{\SO^{+}}

\renewcommand{\so}{\operatorname{\mathfrak{so}}} 

\DeclareMathOperator{\UU}{U}
\DeclareMathOperator{\SU}{SU}
\newcommand{\su}{\mathfrak{su}}
\newcommand{\symp}{\mathfrak{sp}}
\newcommand{\Symp}{\operatorname{Sp}}

\renewcommand{\sl}{\mathfrak{sl}}

\newcommand{\Glie}{\mathfrak g}

\newcommand{\Orb}{\mathcal{O}}

\newcommand{\Killing}{\kappa}

\newcommand{\sprod}[2]{\left\langle #1 \,,\, #2 \right\rangle}

\title{Nilpotent orbits of classical Lie algebras stable under negation}
\date{\today}

\author{Guillaume Neuttiens}
\affil{Friedrich-Schiller Universität Jena, Germany}
\author{Jérémie Pierard de Maujouy}
\affil{Institut Denis Poisson - Université de Tours, France}
\begin{document}

\maketitle

\begin{abstract}
    Gibbs states are probability distributions defined on Hamiltonian $G$-manifolds that are naturally parametrized by elements of the Lie algebra $\Glie$.
    In this paper, we focus on a specific case of the simplest Hamiltonian $G$-manifolds, the coadjoint orbits of Lie algebras. 
    We look at the nilpotent coadjoint orbits of the classical Lie algebras, or equivalently the nilpotent adjoint orbits.
    We show that Gibbs states do not exist on nilpotent orbits that are stable under multiplication by $-1$, and proceed to classify those for all classical Lie algebras.
\end{abstract}

\section{Introduction}

Coadjoint orbits of Lie algebras are well-studied objects, partly due to their role in representation theory, such as Kirillov's orbit method for constructing representations. 
Nilpotent coadjoint orbits are a particular subclass that is simple and easier to handle for several purposes.
In the case of semisimple Lie algebras, the nilpotent coadjoint orbits are isomorphic to the nilpotent adjoint orbits and for the classical Lie algebras can be classified with signed Young tableaux. According to the Jordan decomposition, every adjoint orbit has a nilpotent \enquote{part}, and thus the understanding of nilpotent orbits is an important stepping stone to treat more general adjoint orbits.

Nilpotent adjoint orbits of semisimple Lie algebras are of particular interest in Lie group thermodynamics, which studies the Gibbs set of \enquote{thermodynamic equilbrium states} on Hamiltonian $G$-manifolds. They have a conjectured relation with the theory of homogeneous convex cones and the corresponding symmetric spaces~\cite{VinbergConvexCones}.
It was proved in~\cite{NilpotentOrbitsGSI23} that the Gibbs sets associated with the nilpotent orbits of $\sl(2, \setR)$ are isometric to a direct product of the hyperbolic plane with a line.

The present papers stems from the following observation: the Gibbs set is empty whenever the nilpotent orbit contains opposed points.
This led to the present study of which nilpotent adjoint orbits of semisimple real Lie algebras are stable under multiplication by $-1$.

A general solution to the problem of identifying the Gibbs sets of coadjoint orbits of real Lie algebras was recently provided in~\cite{NeebGibbsEnsembles}. We think nevertheless that the observation at the root of the present paper has value because of its simplicity and could be of independant interest .
Our approach largely follows the method from~\cite{CollingwoodMcGovern}, for which we provide a review.

After the present introduction, Section~\ref{secno:AdjointOrbits} recalls general results on the nilpotent orbits of semisimple Lie algebras, and their consequences for Gibbs states: in particular, we explain why nilpotent orbits stable under negation cannot admit Gibbs states.
In Section~\ref{secno:InvariantProducts}, we review the necessary results regarding invariant products on $\sl(2, \setR)$-modules and deduce their consequences for the classification of the nilpotent adjoint orbits of the classical Lie algebras, according to the methods of~\cite{CollingwoodMcGovern}. This is applied to the problem of comparing the nilpotent orbits of two opposed nilpotent elements.
Section~\ref{secno:ClassicalLieAlgebras} applies the results of Section~\ref{secno:InvariantProducts} to every classical real Lie algebra, and we give necessary and sufficient conditions for a nilpotent orbit to be stable under the negation map.

\section{Nilpotent adjoint orbits of semisimple Lie groups}
\label{secno:AdjointOrbits}

\subsection{Standard triples}

Rather than working with just a nilpotent element of a semisimple Lie algebra, it is often convenient to include the element in a $\sl(2, \setR)$-subalgebra. A reference for this section is~\cite{CollingwoodMcGovern}.

\begin{definition}[Standard triple]
    Let $\Glie$ be a Lie algebra. A \emph{standard triple} is a triplet $(x,y,h)$ of elements of $\Glie$ satisfying the following relations:
    \begin{align*}
        [h,x] &= 2x\\
        [h,y] &= 2y\\
        [x,y] &= h.
    \end{align*}
\end{definition}

In particular, a standard triple in $\Glie$ defines a subalgebra isomorphic to $\sl(2, \setR)$; in fact, a standard triple can be understood as an embedding $\sl(2, \setR)\hookrightarrow \Glie$, once a standard triple $(X, Y, H)$ in $\sl(2, \setR)$ is fixed.
The following property allows using standard triples in the study of nilpotent elements:
\begin{theorem}[Jacobson-Morozov theorem]
\label{thmno:JacobsonMorosov}
    Let $\Glie$ be a real semisimple Lie algebra.
    For any nonzero nilpotent element $x\in \Glie$, there exist $y,h\in \Glie$ such that $(x, y, h)$ is a standard triple for $\Glie$.
\end{theorem}

The conjugacy classes of the standard triple naturally correspond to the conjugacy classes of the nonzero nilpotent elements:
\begin{theorem}
\label{thmno:KostantConjugacy}
    Let $\Glie$ be a real semisimple Lie algebra with an adjoint group $G$.
    If $(x,y,h)$ and $(x,y', h')$ are two standard triples associated to $x$ then they are conjugate under the centralizer $Z_{G}(x)$.
\end{theorem}

\begin{corollary}\label{corno:TriplesConjugate}
    Let $\Glie$ be a real semisimple Lie algebra with an adjoint group $G$.
    Let $(x,y,h)$ and $(x',y', h')$ be two standard triples.
	There exists $g\in G$ such that $g\cdot (x,y,h)= (x',y',h')$ if and only if $x$ and $x'$ are $G$-conjugate.
\end{corollary}

\subsection{Conical structure of nilpotent orbits}

We show that nilpotent orbits of semisimple Lie algebras have a conical structure.
\begin{proposition}\label{propno:NilpotentOrbitsConical}
    ~\begin{enumerate}
        \item Nilpotent coadjoint orbits of a real semisimple Lie algebra are stable under multiplication by $\setR^*_+$, the set of positive real numbers,
        \item Nilpotent coadjoint orbits of a complex semisimple Lie algebra are stable under multiplication by $\setC^*$, the set of invertible complex numbers.
    \end{enumerate}
\end{proposition}

\begin{proof}
    Let $\Glie$ be a real semisimple Lie algebra and $x$ a nilpotent element. According to Theorem~\ref{thmno:JacobsonMorosov} there exists $h\in \Glie$ such that $[h,x]=2x$. For all $t\in \setR$,
    \[
        \exp(t\ad_h)(x) = e^{2t}x   
    \]
    belongs to the coadjoint orbit of $x$. Therefore $\setR^*_+ x$ is contained in the nilpotent orbit of $x$.

    If $\Glie$ admits a complex structure $I$, then for every $z=a+ib\in \setC$, 
    \[
        \exp(a\ad_h + b \ad_{Ih})(x) = e^{2z} x,    
    \]
    which proves the second point.
\end{proof}

\subsection{Gibbs states}

Let $\Glie$ be a semisimple Lie algebra and let $\Killing$ be its Killing form. Given any $X\in \Glie$, we shall write $X^*$ for the associated fundamental vector field on $\Glie$: $X^*|_x = \ad_X x$.
Then any adjoint orbit $\Orb$ of $\Glie$ is a Hamiltonian $G$-manifold with the Kirillov-Kostant-Souriau symplectic form:
\[
    \omega_{\Orb}|_x : 
    \left(
        X^*, Y^*
    \right)
    \mapsto
    \Killing(x, [X, Y]).
\]
The associated Liouville volume form is $\lambda_{\Orb} = \omega_{\Orb}^{\dim \Orb /2}$. 

Let $\beta\in \Glie^*$. The following integral is called the \emph{partition function} of $\Orb$:
\[
    Z(\beta) = \int_{\Orb} e^{-\sprod{\beta}{y}} |\lambda_{\Orb}| (y)
    \in (0, +\infty]
\]
When 
$Z(\beta) < \infty$,
one defines the associated \emph{Gibbs state}~\cite{SSDEng, MarleGibbsStates}, a probability distribution on $\Orb$:
\[
    y \in \Orb \mapsto
    \frac{
    e^{-\sprod{\beta}{y}} |\lambda_{\Orb}| (y)
    }{
    Z(\beta)
    }
.\]
Gibbs states are distributions of \emph{maximal entropy} under a constraint of fixed average value in $\Glie$. In a physical context, they describe states of thermodynamic equilibrium.

We will prove that the partition function of an adjoint orbit that contains opposed elements is never finite, and therefore such an orbit has no Gibbs state.

We show first that the Kirillov-Kostant-Souriau form is homogeneous.

\begin{proposition} 
\label{propno:KKSHomogene}
    For $x\in \Glie$, write $\Orb_x$ the adjoint orbit of $x$ and let $2d$ be its dimension.
    Let $t\in \setR^*$.
    Then $\Orb_{tx} = t\cdot \Orb_x$ and 
    \begin{enumerate}
        \item     \(
        (t\Id)^* \omega_{\Orb_{tx}} = t \omega_{\Orb_x}
                ,\)
        \item 
        \(
           (t \Id)^* \lambda_{\Orb_{tx}} = t^d \lambda_{\Orb_x}
       .\)
    \end{enumerate}
\end{proposition}
\begin{proof}
    The identity $\Orb_{tx} = t\cdot \Orb_x$ is a direct consequence of the linearity of the adjoint action of $G$.
    Let $X, Y, x \in \Glie$. Notice that $t[X, x] = [X, tx]$. Therefore $(t\Id)_*X^* = X^*$ and
    \[
        \omega_{\Orb_tx} (X^*, Y^*) = \Killing(tx, [X,Y]) = t \Killing(x, [X,Y]) = t\omega_{\Orb_x}(X^*, Y^*) 
    ,\]
    which proves the homogeneity of the Kirillov-Kostant-Souriau form. Homogeneity of the Liouville form ensues.
\end{proof}

We can now prove that the partition function of a nilpotent orbit stable under negation always diverges.

\begin{proposition}\label{propno:NoGenTemp}
    Let $x$ be a nonzero nilpotent element in $\Glie$.
    If $x$ and $-x$ are in the same adjoint orbit, then for all $\beta\in \Glie$,
    \[
        \int_{\Orb_x} e^{-\sprod{\beta}{y}} |\lambda_{\Orb_x}| (y)
        = \infty
    .\]
\end{proposition}

The proof will be organized into three short arguments.
\begin{proof}
    ~\begin{enumerate}
        \item Let $\beta\in \Glie$. Then
        \begin{equation*}
            \begin{aligned}
          \int_{\Orb_x} e^{-\sprod{-\beta}{y}} |\lambda_{\Orb_x}| (y)
          &= \int_{\Orb_x} e^{-\sprod{\beta}{-y}} |\lambda_{\Orb_x}| (y)\\
          &= \int_{-\Orb_x} e^{-\sprod{\beta}{y'}} |\lambda_{-\Orb_x}| (y')
            ,\end{aligned}
        \end{equation*}
        with the last line using the homogeneity of $\lambda$ (Propostion~\ref{propno:KKSHomogene}). In particular, when $-\Orb_x = \Orb_x$, then the integral has the same value for the parameters $\beta$ and $-\beta$.

        \item \emph{Nonzero nilpotent orbits have infinite volume.} Let $t\in \setR_+^*$. Then a pullback under the diffeomorphism $t\id : \Orb_x \to \Orb_x$ gives
        \[
            \int_{\Orb_x} |\lambda_{\Orb_x}|
            = \int_{\Orb_x} (t\id)|^*\lambda_{\Orb_x}|
            = t^{\dim \Orb_x / 2} \int_{\Orb_x} |\lambda_{\Orb_x}|
        \quad,\]
        which can only hold for an infinite volume (since the volume cannot be zero).

        \item \emph{Convexity of the integral with respect to $\beta$.}
        Assuming that $-\Orb_x = \Orb_x$, 
        \begin{equation*}
            \begin{aligned}
                  \int_{\Orb_x} e^{-\sprod{\beta}{y}} |\lambda_{\Orb_x}| (y)
                &= \frac12 \int_{\Orb_x} e^{-\sprod{\beta}{y}} + e^{-\sprod{-\beta}{y}} |\lambda_{\Orb_x}| (y)
                \geqslant \int_{\Orb_x} 1 |\lambda_{\Orb_x}| (y)
                = \infty
                .
            \end{aligned}
        \end{equation*}
    \end{enumerate}
\end{proof}

\begin{corollary}
	No nonzero nilpotent adjoint orbit of a complex semisimple Lie algebra admits Gibbs states.
\end{corollary}
\begin{proof}
    According to Proposition~\ref{propno:NilpotentOrbitsConical}, nilpotent adjoint orbits of a complex semisimple Lie algebra are stable under negation, since $-1\in \setC^*$.
\end{proof}

\section{Invariant products on $\sl(2, \setR)$-modules}\label{secno:InvariantProducts}

Most classical Lie algebras are characterized by an invariant bilinear form. According to Theorem~\ref{thmno:KostantConjugacy}, the classification of conjugacy classes of nilpotent elements can be handled at the level of standard triples. Therefore, we are interested into standard triples up to automorphisms of bilinear forms, namely, in isomorphism classes of $\sl(2, \setR)$-modules with a suitable invariant bilinear form. 

We thus need to discuss the invariant bilinear forms on $\sl(2, \setR)$-modules of finite dimension. Our approach reproduces that of~\cite{CollingwoodMcGovern}
\subsection{Irreducible $\sl(2, \setR)$-modules}

Recall that an irreducible $\sl(2, \setR)$-module of finite dimension is characterized up to isomorphism by its dimension. If $(V, \rho)$ is the irreducible real $\sl(2, \setR)$-module of dimension $r+1$, then it decomposes into a direct sum of $1$-dimensional weight spaces
\[
   V = 
    \bigoplus_{i=0}^r V(r-2i)
\]
with $V(k) = \ker (\rho(H)-k\id_V)$.

We recall the following standard fact about invariant products on $\sl(2, \setR)$-modules:
\begin{proposition}\label{propno:InvariantProductssl2modules}
    Let $V$ be an irreducible real $\sl(2, \setR)$-module. The space of invariant bilinear forms on $V$ has dimension $1$. 
    If $V$ has even dimension, the invariant forms are symplectic. 
    If $V$ has odd dimension $2m+1$, the invariant forms are symmetric of signature $(m+1,m)$ (or $(m,m+1)$).
\end{proposition}
We follow the arguments of~\cite[Chapter 5]{CollingwoodMcGovern}.
\begin{proof}
    Invariant bilinear forms on $V$ are in one-to-one correspondance with the $\sl(2, \setR)$-equivariant linear maps $V\to V^*$. Since $V$ and $V^*$ have the same dimension, they are isomorphic representations of $\sl(2, \setR)$. Furthermore, since they are highest weight modules, the space of homomorphisms of representation has dimension $1$. This proves that the invariant bilinear forms on $V$ form a space of dimension $1$. Furthermore, since the symmmetric and antisymmetric parts of an invariant bilinear form are invariant, $V$ is either constituted of symmetric forms or antisymmetric forms.

    Let $B$ be a nonzero invariant bilinear form on $V$. Since $\ker B$ is an invariant submodule, $B$ is nondegenerate.
    Define $r:=\dim V-1$ and,
    for $\lambda\in \setR$, define $V(\lambda) := \ker (\rho(H)-\lambda \id_V)$ with $\rho$ the representation of $\Glie$ on $V$.
    Let $a\in V(\lambda)$ and $b\in V(\mu)$. Then
    \[
        0
            = B(\rho(H)a, b) + B(a, \rho(H)b)
            = B(\lambda a, b) + B(a, \mu b)
            = (\lambda + \mu) B(a,b)
    \]
    so that $\lambda + \mu \neq 0 \implies V(\lambda) \perp_B V(\mu)$. 

    Let $a\in V({r})$ be a nonzero highest weight vector. Then $\rho(y)^r a$ is a nonzero lowest weight vector, so that 
    \[
        0 
        \neq B(a, \rho(y)^r a) 
        = B((-\rho(y))^r a, a) 
        = (-1)^r B(\rho(y)^r a, a) 
        .
    \]
    Since $B$ is either symmetric or antisymmetric, if $r$ is even, $B$ is symmetric, and if $r$ is odd, $B$ is antisymmetric.

    Assume $r$ is even: $B$ is symmetric. Then $V$ decomposes into an orthogonal sum of hyperbolic vector planes and a vector line
    \[
        V = 
            \bigoplus_{1 \leqslant i \leqslant \frac{r}2} 
                {}^{\hspace{-1em}\perp} \;
                (V({-2i}) \oplus V({2i})) \oplus_\perp V(0)
    \]
    Each hyperbolic plane $V({-2i}) \oplus V({2i})$ has signature $(1,1)$, and the signature of $V(0)$ can be either positive or negative, therefore the global signature of $B$ is either $(m,m+1)$ or $(m+1, m)$. 
\end{proof}

For every $r \geqslant 0$, we fix an irreducible $\sl(2, \setR)$-module $(V_r, \rho_r)$ of dimension $r+1$, as well as a non-degenerate bilinear form $B_r$ on $V_r$.
When $r=2m$, we choose $B_i$ to have signature $(m+1, m)$ (namely, $V(0)$ is always a positive subspace). 
When $r$ is odd, we choose $B_i$ such that $\forall v\in V_r({-1}),\, B_i(v, X\cdot v) \geqslant 0$ with $X = \begin{pmatrix}
	0 & 1 \\ 0 & 0
\end{pmatrix}$.

In order to compare the adjoint orbits of $x$ and $-x$, we will need an element that conjugates them:
\begin{proposition}\label{propno:x->-x}
    Let $(x,y,h)$ be a standard triple in $\sl(2, \setR)$.
        There exists an invertible endomorphism $u_r: V_r\to V_r$ such that 
        \begin{align}
            u_r \rho_r(x) u^{-1}_r &= -\rho_r(x)\\
            u_r \rho_r(y) u^{-1}_r &= -\rho_r(y)\\
            u_r \rho_r(h) u^{-1}_r &= \rho_r(h)
        \end{align}
        and for any $\sl(2, \setR)$-invariant bilinear form $B$, $u_r^*B = (-1)^r B$.
\end{proposition}
\begin{proof}
    We decompose $V$ in weight spaces for $h$:
    \[
        V = \bigoplus_{i=0}^r V({r-2i})
    \]
    and define $u_r = \bigoplus_i (-1)^{i} \id_{V({r-2i})}$.
    It satisfies 
    $u_r \rho_r(x) u_r^{-1} = -\rho_r(x)$, 
    $u_r \rho_r(y) u^{-1}_r = -\rho_r(y)$ and 
    $u_r \rho_r(h) u^{-1}_r = \rho_r(h)$. 
    Recall that the invariant product on $V$ couples $V({r-2i})$ with $V({2i-r})$ and that weight spaces of non-opposite weights are orthogonal.
    For any $i$, $2i-r = r -2(r-i)$ therefore $u_r$ acts on
    $
        V({r-2i})\otimes V({2i-r})
    $
    by multiplication with $(-1)^{i+r} (-1)^{2r-i} = (-1)^r$.

    If $r$ is even, $u_r$ therefore preserves any invariant bilinear form.
    If $r$ is odd, $u_r$ defines a congruence between any invariant symplectic form and its opposite: $u_r^*B = -B$ for any invariant symplectic form $B$.
\end{proof}

\begin{remark}
    For even $r$, we will use later the alternative choice $\tilde u_r = (-1)^{r/2}\bigoplus_i (-1)^{i} \id_{V(r-2i)}$, that has the benefit of having determinant $1$.
\end{remark}

\subsection{General $\sl(2, \setR)$-modules with invariant product}

We will handle $\sl(2, \setR)$-modules using the following form of isotypic decomposition:
\begin{proposition}\label{propno:sl2modKSpace}
	Let $V$ be a finite dimensional $\sl(2, \setR)$-module equipped with a linear action of $\setK\in \{\setR, \setC, \setH\}$ that commutes with the action of $\sl(2, \setR)$. Then $V$ is $\setK$-linearly isomorphic to an $\sl(2, \setR)$-module of the form
	\[
		\bigoplus_r V_r \otimes_\setR W_r
	\]
	with $V_r$ the irreducible real $\sl(2, \setR)$-module of dimension $r+1$ and $W_r$ a $\setK$-vector space (with trivial action of $\sl(2, \setR)$).

    If $V'$ is another $\setK$-linear $\sl(2, \setR)$-module of the form
    \[
        \bigoplus_r V_r \otimes_\setR W'_r
    \]
    then any $\setK$-linear map of $\sl(2, \setR)$-modules $f:V\to V'$ takes the form
    \[
        \bigoplus_r \id_{V_r} \otimes f_r
    \]
    with $f_r\in \Hom_{\setK \hVect}(W_r, W'_r)$.
\end{proposition}

\begin{proof}
    Since $\sl(2, \setR)$ is a semisimple Lie algebra, $V$ decomposes into isotypic components. 
	Let $V_r$ be an irreducible $\sl(2, \setR)$-module. The corresponding isotypic component $C_r$ of $V$ is a $\setK$-vector subspace. $C_r$ factorizes as $V_r\otimes_\setR W_r$ as a real $\sl(2, \setR)$-module. Furthermore, Schur's Lemma implies that the scalar action of $\setK$ that commutes with the action of $\sl(2, \setR)$ is naturally represented by real automorphisms of $W_r$, thereby defining a $\setK$-vector space structure on $W_r$.

    Finally, such a map $f$ maps the isotypic components of $V$ into isotypic components of $V'$ of the same type. Since the $V_r$ are irreducible, Schur's Lemma for $\setK$-linear $\sl(2, \setR)$-modules implies that $f$ decomposes into the given form.
\end{proof}
\begin{remark}
    A canonical choice for $W_r$ is given by $\Hom_{\sl(2, \setR) \hMod }(S_r, V)$, but this expression will not be of use to us.
\end{remark}%
We will call \enquote{multiplicity spaces} the spaces $W_r$.

In order to handle non-irreducible $\sl(2, \setR)$, we will use tensors products of bilinear forms.
\begin{proposition}\label{propno:TensorProduct}
    Let $W_1$ be a $\setR$-vector space equipped with a bilinear form $B_1$ and
    let $W_2$ be a $\setK$-vector space equipped with a bilinear (resp. sesquilinear~\footnote{Our convention is left-antilinear and right-linear.}) form $B_2$. 
    The tensor product $W_1\otimes_\setR W_2$ is a $\setK$-vector space, equipped with the tensor product bilinear (resp. sesquilinear) form $B_1\otimes B_2$ defined as follows on decomposable vectors:
    \[
        B_1\otimes B_2(a_1\otimes a_2, b_1\otimes b_2)
        := B_1(a_1, b_1)B_2(a_2, b_2).
    \]

    \begin{enumerate}
    	\item If $B_1 \neq 0$ then for any bilinear form $B'_2$ on $W_2$ such that $B_1\otimes B'_2 = B_1\otimes B_2$ it holds $B'_2 = B_2$.

        \item Assume that that $B_1, B_2\neq 0$. 
        $B_1\otimes B_2$ is $\setK$-bilinear (resp. sesquilinear) if and only $B_2$ is $\setK$-bilinear (resp. sesquilinear).

        In each case, the symmetry property of two among $B_1, B_2, B_1\otimes B_2$ implies a symmetry property on the third one, according to the following tables:
        \begin{gather*}
            \begin{NiceTabular}{|c|c|c|}
                \hline
                \multicolumn{3}{c}{$\setK$-bilinear case: symmetry of $B_1\otimes B_2$}
                 \\ \hline \
                 \diagbox{$B_1$}{$B_2$} & symmetric      & skew-symmetric\\ \hline
                 symmetric                 & symmetric     & skew-symmetric \\ \hline
                 skew-symmetric           & skew-symmetric  & symmetric \\ \hline
            \end{NiceTabular}
\\[1em]
            \begin{NiceTabular}{|c|c|c|}
                \hline
                \multicolumn{3}{c}{$\setK$-sesquilinear case: symmetry of $B_1\otimes B_2$}
                 \\ \hline \
                 \diagbox{$B_1$}{$B_2$} & Hermitian      & skew-Hermitian\\ \hline
                 symmetric              & Hermitian      & skew-Hermitian \\ \hline
                 skew-symmetric         & skew-Hermitian & Hermitian \\ \hline
            \end{NiceTabular}
        \end{gather*}
    \end{enumerate}
\end{proposition}
\begin{proof}
	Since $B_1 \neq 0$, there exist $u,v\in W_1$ such that $B_1(u_1, v_1)\neq 0$. This gives the following expression for $B_2$:
	    \begin{equation}\label{eqno:B1B2/B1}
	        B_2(x,y) = \frac{B_1\otimes B_2(u\otimes x, v\otimes y)}{B_1(u,v)}
	    .\end{equation}
	Therefore, $B_2$ is uniquely characterized by $B_1\otimes B_2$ and $B_1$.
    
	Moreover, Equation~\eqref{eqno:B1B2/B1} implies immediately that $B_2$ and $B_1\otimes B_2$ have the same $\setK$-bilinearity/ sesquilinearity. The symmetry properties are also immediate.
\end{proof}

The following Proposition will be our main tool to characterize nilpotent adjoint orbits of simple Lie groups.

\begin{proposition}\label{propno:sl2moduleIsometry}
    Let $V$ be a finite-dimensional $\setK$-linear $\sl(2, \setR)$-module equipped with an invariant bilinear (resp. sesquilinear) form $B$.
    \begin{enumerate}
    \item 
    There is an isomorphism of $\sl(2, \setR)$-modules
    \[
        V \simeq \bigoplus_r V_r \otimes W_r
    \]
    with $V_r$ the fixed irreducible $\sl(2, \setR)$-module of dimension $r+1$ and $\setK$-vector spaces $W_r$.
    Under this isomorphism, $B$ takes the form
    \[
        B \simeq \bigoplus B_r \otimes \phi_r
    \]
    with $B_r$ the fixed invariant bilinear form on $V_r$ and $\phi_r$ an arbitrary bilinear (resp. sesquilinear) form on $W_r$.

    \item     
    Let $V'$ be another $\setK$-linear $\sl(2, \setR)$-module equipped with an invariant bilinear (resp. sesquilinear) form $B'$, of the form
    \[
    	V' \simeq \bigoplus_r (V_r\otimes W'_r, \, B_r \otimes \phi'_r)
	.\]
	Then any $\setK$-linear map of $\sl(2, \setR)$-modules $f:V\to V'$ such that $u^*B' = B$ takes the form
    \[
    	f \simeq \bigoplus_r \id_{V_r} \otimes f_r
    \]
    with $f_r\in \Hom_{\setK \hVect}(W_r, W'_r)$ such that $f_r^*\phi'_r = \phi_r$.
    \end{enumerate}
\end{proposition}
The argument is a reformulation of~\cite{CollingwoodMcGovern}.
\begin{proof}
    For simplicity, we write the proof only for the case where $B$ is 
    sesquilinear; 
    the other cases are essentially identical.
    Since $\sl(2, \setR)$ is semisimple, $V$ decomposes into isotypic components
    \[
        V \simeq_{\sl(2, \setR)} \bigoplus_r V_r\otimes W_r 
    \]
    with $\sl(2, \setR)$ acting only on the $V_r$ factors, the action on $W_r$ being trivial.
    Since the $V_r$ are self-dual $\sl(2, \setR)$-modules, the different isotypic components are pairwise $B$-orthogonal. Therefore, $B$ decomposes as a direct sum $B \simeq \bigoplus_r \widetilde{B}_r$ with $\widetilde{B}_r$ a sesquilinear form on $V_r\otimes W_r$.

    Furthermore, Schur's Lemma implies that the invariant form $\widetilde{B}_r$ takes the form $B_r\otimes \phi_r$ with $B_r$ an invariant bilinear form on $V_r$ and $\phi_r$ an arbitrary sesquilinear form on $W_r$. 
    Indeed, $\widetilde{B}_r$ can be interpreted as a $\sl(2, \setR)$-equivariant $\setK$-antilinear map $V_r\otimes W_r \to V_r^*\otimes W_r^*$. Since $V_r$ and $V_r^*$ are isomorphic (real) irreducible $\sl(2, \setR)$-modules, Schur's Lemma implies that all such maps are of form $B_r \otimes \phi_r$ with $\phi_r$ a sesquilinear form on $W_r$.
    
    Finally, let $f\in \Hom_{\sl(2, \setR) \hMod} (V, V')$ such that $f^* B' = B$. According to Schur's Lemma (Proposition~\ref{propno:sl2modKSpace}), $f$ takes the form
    \[
    	f \simeq \bigoplus_r \id_{V_r} \otimes f_r
    \]
    with $f_r: W_r \to W'_r$. Since $f$ maps $B$ to $B'$, necessarily, $\id_{V_r}\otimes f_r$ maps $B_r \otimes \phi_r$ to $B_r\otimes \phi'_r$. Thus
    \[
    	(\id_{V_r}\otimes f_r)^* (B_r \otimes \phi'_r)
    	= 
    	(\id_{V_r}^* B_r)\otimes (f_r^* \phi'_r)
     	= 
     	B_r \otimes (f_r^* \phi'_r)
    .\]
    Since $B_r$ is nonzero, Proposition~\ref{propno:TensorProduct} implies that $f_r^*\phi'_r = \phi_r$.
\end{proof}

Using the maps $u_r$ introduced in Proposition~\ref{propno:x->-x}, we can now compare in generality the adjoint orbits of a nilpotent element $x$ with that of $-x$.

\begin{theorem}\label{thmno:IsometryConjugacy}
	Let $V$ be a $\setK$-vector space with $\setK \in \{\setR, \setC, \setH\}$. 
	Assume that $V$ is equipped with a non-degenerate bilinear symmetric, symplectic, Hermitian, or skew-Hermitian form $B$. Define the group
	\[
		G = \left\{
			g \in \GL_\setK(V) \, | \,
			g^*B = B
		\right\}
	\]
	and call its Lie algebra $\Glie$.
	
\begin{enumerate}
    \item $G$ is semisimple and thus every nonzero nilpotent element belongs to a standard triple.

	\item 
	Let $(x,y,h)$ be a standard triple in $\Glie$. It defines an action of $\sl(2, \setR)$ on $V$ through the embedding $\sl(2, \setR) \hookrightarrow \Glie$. Consider the isotypic decomposition of $V$. Then the $\setK$-linear isometry classes of $(W_r, \phi_r)$ depend only on $x$ (and not on the choice of $y$ and $h$).
	
	\item Two nilpotent elements in $\Glie$ are $G$-conjugate if and only if for every $r$ they have $\setK$-linearly isometric multiplicity spaces $(W_r, \phi_r)$.

	\item Any nonzero nilpotent element $x$ is $G$-conjugate to $-x$ if and only if for every odd $r$, $(W_r, \phi_r)$ is $\setK$-linearly isometric to $(W_r, -\phi_r)$.
\end{enumerate}	
\end{theorem}

\begin{proof}
~\begin{enumerate}[leftmargin=*]
	\item 
	First, note that the Lie group $G$ is semisimple (see any reference on the classical Lie groups such as~\cite{Helgason}).
	Therefore, Theorem~\ref{thmno:JacobsonMorosov} implies that every nonzero nilpotent element of $\Glie$ belongs to a standard triple. 
    
    \item 
    Theorem~\ref{thmno:KostantConjugacy} implies that two standard triples $(x,y,h)$ and $(x,y',h')$ in $\Glie$ are $G$-conjugate. According to Proposition~\ref{propno:sl2moduleIsometry}, this implies that the isometry classes of $(W_r, \phi_r)$ are the same for the two triples and thus only depend on $x$.
	
	\item 
	Second, the same argument, along with Corollary~\ref{corno:TriplesConjugate} implies that two nonzero nilpotent elements are conjugate if and only if they have the same isometry classes of $(W_r, \phi_r)$.
	
	\item We shall be precise and write $\rho_r$ for the representation of $\sl(2, \setR)$ on $V_r$.
	Let $(x,y,h)$ be a standard triple in $\Glie$: the corresponding action $\rho$ of $\sl(2, \setR)$ on $V$ satisfies
	\begin{equation*}
		\rho(X) = x, \qquad\qquad
		\rho(Y) = y, \qquad\qquad
		\rho(H) = h
	.\end{equation*}
	Consider the isotypic decomposition
	\[
		(V, \rho) \simeq_{\sl(2, \setR)} \bigoplus_r (V_r, \rho_r) \otimes_\setR W_r
	\]
	with each $W_r$ a $\setK$-linear vector space equipped with a bilinear form $\phi_r$.
	
	Consider now the standard triple $(-x,-y,h)$ and the associated action $\tilde \rho$ of $\sl(2, \setR)$ on $V$:
	\begin{equation*}
		\tilde\rho(X) = -x, \qquad\qquad
		\tilde\rho(Y) = -y, \qquad\qquad
		\tilde\rho(H) = h
	.\end{equation*}
	If we define, in a similar way, alternative representations $\tilde\rho_r$ on each space $V_r$, we obtain an alternative isotypic decomposition
	\[
		(V, \tilde\rho) \simeq_{\sl(2, \setR)} \bigoplus_r (V_r, \tilde \rho_r) \otimes_\setR W_r
	.\]
	We are looking for the isotypic decomposition of $(V, \tilde\rho)$ \emph{in terms of the representatives $(V_r, \rho_r)$}.
	
	In Proposition~\ref{propno:x->-x}, we have constructed for every $r$ a map $u_r: V_r \to V_r$ such that 
	\[
		u_r
			\big(\rho_r(x), \rho_r(y), \rho_r(h)
			\big) u_r^{-1}
		= \big( 
			-\rho_r(x), -\rho_r(y), \rho_r(h) \big)
		= \big(
			\tilde\rho_r(x), \tilde\rho_r(y), \tilde\rho_r(h) \big)
	\]
	and $u_r^*B_r = (-1)^r B_r$. Therefore, we can construct homomorphisms of $\sl(2, \setR)$-modules
	\[
		(V_r, \rho_r) \otimes_\setR W_r 
			\xrightarrow[\sim]{u_r \otimes \id_{W_r}}
		(V_r, \tilde\rho_r) \otimes_\setR W_r
	\]
	with $(u_r^* B_r)\otimes \phi_r = (-1)^r B_r \otimes \phi_r = B_r \otimes (-1)^r \phi_r$. A direct sum gives an isomorphism of $\sl(2, \setR)$-modules that is compatible with the bilinear forms:
	\begin{equation*}
		\bigoplus_r 
			\Big( 
			(V_r, \tilde \rho_r) \otimes_\setR W_r
			, \,
			B_r\otimes (-1)^r \phi_r 
			\Big)
		\simeq_{\sl(2, \setR)}
			\Big( 
			(V_r, \rho_r) \otimes_\setR W_r
			, \,
			B_r\otimes \phi_r 
			\Big)
		\simeq \Big((V, \tilde\rho), B \Big)		
	.\end{equation*}
	We have proven that the multiplicity spaces associated with $-x$ have isometry classes $(W_r, (-1)^r\phi_r)$. Proposition~\ref{propno:sl2moduleIsometry} allows us to conclude that $x$ and $-x$ are $G$-conjugate if and only if for every $r$ the bilinear form $\phi_r$ is isometric to $(-1)^r \phi_r$, namely, if and only if for every odd $r$ the bilinear $\phi_r$ is isometric to $-\phi_r$.
\end{enumerate}
\end{proof}

\section{Nilpotent orbits of the classical Lie algebras}
\label{secno:ClassicalLieAlgebras}

With the results from Section~\ref{secno:InvariantProducts}, we can identify which nilpotent orbits of the classical Lie groups have opposed points.
The different classical algebras are handled in a mostly uniform fashion, using their standard representation, according to the approach from~\cite{CollingwoodMcGovern}. While their classification is given in the combinatorial language of signed tableaux, it is more practical for us to use the (equivalent) signatures of the multiplicity spaces $(W_i, \phi_i)$.

Recall that the adjoint orbits of Lie algebras are the orbits under the action of \emph{connected} Lie groups on their tangent Lie algebras.
Lie algebras with similar considerations will be presented next to each other, with the complicated case of $\so(p,q, \setR)$ in last.

\subsection{$\sl(n, \setR)$}

The case of $\sl(n, \setR)$ is handled in \cite{SLnGeneralizedTemperature}, which we cite here:
\begin{proposition}
	Let $x$ be a nilpotent element in $\sl(n, \setR)$.
    Then $x$ and $-x$ belong to different $\sl(n, \setR)$-orbits if and only if all Jordan blocks of $x$ have a dimension $n$ such that $n \equiv 2 \mod 4$.

    When $x$ has a Jordan block of dimension $\not\equiv 2 \mod 4$, its adjoint orbit admits no Gibbs state.
\end{proposition}
Using the highest weights $r=n-1$ associated to standard triples as in Section~\ref{secno:InvariantProducts}, the condition becomes $r \equiv 1 \mod 4$.

\subsection{$\symp(p,q)$}

The group $\Symp(p,q)$ is the isotropy group of a Hermitian product $B$ of signature $(p,q)$ on $\setH^{p+q}$. It is connected.
Let us recall the law of inertia (skew-)Hermitian products on quaternionic vector spaces~\cite{QuaternionLinearAlgebra}:
\begin{proposition}\label{propno:QuaternionHermitian}
	Non-degenerate Hermitian products on $\setH^n$ are characterized by a signature $(p,q)$ with $p+q = n$. 
	Non-degenerate skew-Hermitian products on $\setH^n$ are all equivalent.
\end{proposition}

Let $(x, y, h)$ be a standard triple in $\symp(p,q)$. $\setH^{p+q}$ has a corresponding isotypic decomposition
\[
	(\setH^{p+q}, B) \simeq \bigoplus_r (V_r \otimes_\setR W_r, B_r \otimes \phi_r)
\]
with $W_r$ a quaternionic vector space. According to Proposition~\ref{propno:TensorProduct}:
\begin{itemize}
	\item if $r$ is odd, then $\phi_r$ is a non-degenerate skew-Hermitian product, and the isometry class of $(W_r, \phi_r)$ is characterized by its dimension;
	\item if $r$ is even, then $\phi_r$ is a non-degenerate Hermitian product, and the isometry class of $(W_r, \phi_r)$ is characterized by its signature.
\end{itemize}

\begin{theorem}\label{thmno:NegationSpqn}
	Any nilpotent orbit of $\symp(p,q)$ is stable under negation.
	
	In particular, no nonzero nilpotent orbit admits Gibbs states.
\end{theorem}

\begin{proof}
	Let $(x,y,h)$ be a standard triple in $\symp(p,q)$ and $(W_r, \phi_r)$ the associated multiplicity spaces.
	For every odd $r$, $\phi_r$ is non-degenerate and skew-Hermitian, therefore, it is isometric to $-\phi_r$. 
	According to Theorem~\ref{thmno:IsometryConjugacy}, $-x$ is conjugate to $x$, which proves the first part.

	The absence of Gibbs states is an application of Proposition~\ref{propno:NoGenTemp}.
\end{proof}

\subsection{$\symp(2n, \setR)$}

The group $\Symp(2n, \setR)$ is the group of linear automorphisms of $\setR^{2n}$ that preserve a given symplectic product $B$.
It is connected~\cite{Helgason}.

Let $(x,y,h)$ be a standard triple in $\symp(2n,\setR)$ and the associated isotypic decomposition of $\setR^{2n}$:
\[
	(\setR^{2n}, B) 
	\simeq
	\bigoplus_r
		(V_r \otimes W_r, B_r \otimes \phi_r)
\]
with $W_r$ real vector spaces. According to Proposition~\ref{propno:TensorProduct}:
\begin{itemize}
	\item if $r$ is odd, then $\phi_r$ is a non-degenerate symmetric product and the isometry class of $(W_r, \phi_r)$ is determined by the signature $(p_r, q_r)$ of $\phi_r$,
	\item if $r$ is even, then $\phi_r$ is a symplectic product and the isometry class of $(W_r, \phi_r)$ is determined by the dimension of $W_r$.
\end{itemize}

\begin{theorem}\label{thmno:NegationSp2n}
	Let $(x,y,h)$ be a standard triple in $\symp(2n,\setR)$.
	Then the adjoint orbit of $x$ is stable under negation if and only if for every odd $r$, $\phi_r$ has a split signature ($p_r = q_r$).
	
	In this case, the adjoint orbit admits no Gibbs state.
\end{theorem}

\begin{proof}
	Let $(W_r, \phi_r)$ the multiplicity spaces associated with $(x,y,h)$.
	For every odd $r$, $\phi_r$ is non-degenerate and symmetric. According to Theorem~\ref{thmno:IsometryConjugacy}, $-x$ is conjugate to $x$ if and only if $\phi_r$ is isometric to $-\phi_r$, which is equivalent to $\phi_r$ having split signature.

	The absence of Gibbs state is then a consequence of Proposition~\ref{propno:NoGenTemp}.
\end{proof} 

\subsection{$\so^*(2n)$}

The group $\SO^*(2n)$ is the isotropy group of a non-degenerate skew-Hermitian product $B$ on $\setH^n$ (unique up to equivalence, according to Proposition~\ref{propno:QuaternionHermitian}). It is connected~\cite{Helgason}.

Let $(x, y, h)$ be a standard triple in $\so^*(2n)$. $\setH^{n}$ has a corresponding isotypic decomposition
\[
	(\setH^{n}, B) \simeq \bigoplus_r (V_r \otimes_\setR W_r, B_r \otimes \phi_r)
\]
with $W_r$ quaternionic vector spaces. According to Proposition~\ref{propno:TensorProduct}:
\begin{itemize}
	\item if $r$ is odd, then $\phi_r$ is a non-degenerate Hermitian product, and the isometry class of $(W_r, \phi_r)$ is characterized by its signature $(p_r, q_r)$;
	\item if $r$ is even, then $\phi_r$ is a non-degenerate skew-Hermitian product, and the isometry class of $(W_r, \phi_r)$ is characterized by the dimension of $W_r$.
\end{itemize}

\begin{theorem}
	Let $(x, y, h)$ be a standard triple in $\so^*(2n)$.
	
	Then the adjoint orbit of $x$ is stable under negation if and only if for every odd $r$, $(\phi_r)$ has a split signature.
	
	In particular, such nilpotent orbits admit no Gibbs state.
\end{theorem}

\begin{proof}
	Let $(W_r, \phi_r)$ the multiplicity spaces associated with $(x,y,h)$.
	For every odd $r$, $\phi_r$ is non-degenerate and Hermitian. 
	According to Theorem~\ref{thmno:IsometryConjugacy}, $-x$ is conjugate to $x$ if and only if $\phi_r$ is isometric to $-\phi_r$, which is equivalent to $\phi_r$ having split signature.

	The absence of Gibbs state is an application of Proposition~\ref{propno:NoGenTemp}.
\end{proof}

\subsection{$\su(p,q)$}

The group $\SU(p,q)$ is the group of determinant $1$, $\setC$-linear isometries of a Hermitian product $B$ of signature $(p,q)$ on $\setC^{p+q}$.
It is connected, and its Lie algebra $\su(p,q)$ is non-compact as soon as both $p$ and $q$ are nonzero~\cite{Helgason}.
Since we are interested in the nilpotent adjoint orbits, we shall assume below that $p\neq 0$ and $q\neq 0$.

Let $(x, y, h)$ be a standard triple in $\su(p,q)$. $\setC^{p+q}$ has a corresponding isotypic decomposition
\[
	(\setC^{p+q}, B) \simeq \bigoplus_r (V_r \otimes_\setR W_r, B_r \otimes \phi_r)
\]
with $W_r$ complex vector spaces. According to Proposition~\ref{propno:TensorProduct}:
\begin{itemize}
	\item if $r$ is odd, then $\phi_r$ is a non-degenerate skew-Hermitian product, and the isometry class of $(W_r, \phi_r)$ is characterized by the signature $(p_r, q_r)$ of the Hermitian form $-i\phi_r$;
	\item if $r$ is even, then $\phi_r$ is a non-degenerate Hermitian product, and the isometry class of $(W_r, \phi_r)$ is characterized by its signature $(p_r, q_r)$.
\end{itemize}

\begin{proposition}\label{propno:ConjugateUpq}
	Let $x$ be a nonzero nilpotent element of $\su(p,q)$. 
	Then $-x$ is $\UU(p,q)$-conjugate to $x$ if and only if for every even $r$, $\phi_r$ has split signature ($p_r=q_r$).
\end{proposition}
\begin{proof}
	We apply Theorem~\ref{thmno:IsometryConjugacy} with $B$ on $\setC^{p,q}$.
	The group $G$ is $\UU(p,q)$.
	For every odd $r$, the skew-Hermitian product $\phi_r$ is isometric to $-\phi_r$ if and only they have the same skew-Hermitian signature, namely if and only if the Hermitian forms $-i\phi_r$ and $i\phi_r$ have the same signature.
\end{proof}

Proposition~\ref{propno:ConjugateUpq} gives a result of conjugacy under $\UU(p,q)$.
In order to obtain conjugacy under the subgroup $\SU(p,q)$, we use the following property:
\com{Pas trouvé de référence...}
$\SU(p,q) \cdot Z(\UU(p,q)) = \UU(p,q)$.
It implies that the conjugacy classes under $\UU(p,q)$ are exactly the conjugacy classes under the subgroup $\SU(p,q)$.
We can now conclude:
\begin{theorem}\label{thmno:NegationSUpq}
	Let $(x,y,h)$ be a standard triple in $\su(p,q)$.
	Then the adjoint orbit of $x$ is stable under negation if and only if for every odd $r$, $-i\phi_r$ has a split signature ($p_r = q_r$).
	
	In this case, the adjoint orbit admits no Gibbs state.
\end{theorem}

\begin{proof}
	Let $(W_r, \phi_r)$ the multiplicity spaces associated with $(x,y,h)$.
	For every odd $r$, $\phi_r$ is non-degenerate and skew-Hermitian. 
	According to Theorem~\ref{thmno:IsometryConjugacy}, $-x$ is conjugate to $x$ if and only if $\phi_r$ is isometric to $-\phi_r$, which is equivalent to $-i\phi_r$ having split signature.

	The absence of Gibbs state is then a consequence of Proposition~\ref{propno:NoGenTemp}.
\end{proof}

\subsection{$\sl(n, \setH)$}

First, we recall the definition of $\sl(n, \setH)$: for any algebra embedding $\setC \hookrightarrow \setH$, 
\[
	\SL(n, \setH) = \GL(n,\setH) \cap \SL(2n, \setC) = \GL(n,\setH)\cap \SL(4n, \setR)
\]
and similarly at the Lie algebra level. $\SL(n, \setH)$ is connected~\cite{Helgason}.

In this case, there is no invariant bilinear form. According to Proposition~\ref{propno:sl2modKSpace}, the isotypic decomposition of $\setH^n$ takes the form
\[
	\setH^n \simeq \bigoplus_r V_r \otimes_\setR W_r
\]
with $W_r$ quaternionic vector spaces.

\begin{proposition}\label{propno:ConjugateGLnH}
	Let $x$ be a nonzero nilpotent element of $\sl(n, \setH)$. Then $-x$ is always $\GL(n,\setH)$-conjugate to $x$.
\end{proposition}
	We need a simpler version of Theorem~\ref{thmno:IsometryConjugacy}, for $G=\GL(n,\setH)$, which we briefly prove here.
\begin{proof}
	As in the proof of Theorem~\ref{thmno:IsometryConjugacy}, we shall use the maps 
	\[
		(V_r, \rho_r) \otimes_\setR W_r 
			\xrightarrow[\sim]{u_r \otimes \id_{W_r}}
		(V_r, \tilde\rho_r) \otimes_\setR W_r
	\]
	that are $\setH$-linear and conjugate the actions of $x$ and $-x$.
	The endomorphism $\bigoplus_r u_r \otimes \id_{W_r}$ is a map in $\GL(\setH^n)$ that conjugates $x$ and $-x$.
\end{proof}

Proposition~\ref{propno:ConjugateGLnH} gives a result of conjugacy under $\GL(n,\setH)$.
In order to obtain conjugacy under the subgroup $\SL(n, \setH)$, we use the following property:
\com{Pas trouvé de référence...}
$\SL(n, \setH) \cdot Z(\GL(n,\setH)) = \GL(n,\setH)$.
It implies that conjugacy classes under $\GL(n,\setH)$ are exactly conjugacy classes under the subgroup $\sl(n, \setH)$.
We can now conclude:
\begin{theorem}\label{thmno:NegationSLnH}
	Any nilpotent orbit of $\sl(n, \setH)$ is stable under negation.
	
	In particular, no nonzero nilpotent orbit admits Gibbs states.
\end{theorem}

\begin{proof}
	The absence of Gibbs states is an application of Proposition~\ref{propno:NoGenTemp}.
\end{proof}

\subsection{$\so(p,q, \setR)$}

The group $\O(p,q, \setR)$ is the group of isometries of a symmetric product $B$ of signature $(p,q)$ on $\setR^{p+q}$.
The Lie-algebra is non-compact as soon as both $p$ and $q$ are nonzero, in which case 
$\O(p,q, \setR)$ has four connected components.
Since we are interested in the nilpotent adjoint orbits, we shall assume below that $p\neq 0$ and $q\neq 0$,
and call the identity component $\SOp(p,q, \setR)$
and $\so(p,q, \setR)$ its Lie algebra.

Let $(x, y, h)$ be a standard triple in $\so(p,q, \setR)$. $\setR^{p+q}$ has a corresponding isotypic decomposition
\[
	(\setR^{p+q}, B) \simeq \bigoplus_r (V_r \otimes_\setR W_r, B_r \otimes \phi_r)
\]
with $W_r$ real vector spaces. According to Proposition~\ref{propno:TensorProduct}:
\begin{itemize}
	\item if $r$ is odd, then $\phi_r$ is a symplectic product, and the isometry class of $(W_r, \phi_r)$ is characterized by the dimension $n_r$ of $W_r$;
	\item if $r$ is even, then $\phi_r$ is a non-degenerate symmetric product, and the isometry class of $(W_r, \phi_r)$ is characterized by its signature $(p_r, q_r)$.
\end{itemize}

\begin{proposition}\label{propno:ConjugateOpq}
	Let $x$ be a nonzero nilpotent element of $\so(p,q, \setR)$. 
	Then $-x$ is always $\O(p,q,\setR)$-conjugate to $x$.
\end{proposition}
\begin{proof}
	We apply Theorem~\ref{thmno:IsometryConjugacy} with $B$ on $\setR^{p,q}$.
	The group $G$ is $\O(p,q,\setR)$.
	For every odd $r$, the symplectic product $\phi_r$ is always isometric to $-\phi_r$.
\end{proof}

Proposition~\ref{propno:ConjugateOpq} gives a result of conjugacy under $\O(p,q,\setR)$.
The subtlety of the case $\so(p,q, \setR)$ is that the isometry group $\O(p,q,\setR)$ is not connected: it has four connected components, thus $\O(p,q,\setR)$-orbits are potentially larger that the $\SOp(p,q, \setR)$-orbits.
More work is needed to characterize the $\SOp(p,q, \setR)$-conjugacy orbits.

Our analysis will rely on the following standard Lemma:
\begin{lemma}\label{lmno:SOpqConnectedComponents}
	Let $x\in \so(p,q, \setR)$ and $g\in \O(p,q,\setR)$.
	Then $x$ and $\Ad_g x$ are in the same $\SOp(p,q,\setR)$-conjugacy orbit if and only if $g\in \SOp(p,q,\setR)\cdot Z_{\O(p,q,\setR)}(x)$, namely, if there exists an element of $Z_{\O(p,q,\setR)}(x)$ in the same connected component as $g$.
\end{lemma}
\begin{proof}
	The first part is straightforward:
	\begin{equation*}
	\begin{aligned}
		\forall h \in \SOp(p,q,\setR), 
		\Ad_g x = \Ad_h x
		\Leftrightarrow
		\Ad_{h^{-1}g} x = x
		\Leftrightarrow
		h^{-1}g \in Z_{\O(p,q,\setR)}(x)
	\end{aligned}
	\end{equation*}
	so the existence of such an $h$ in $\SOp(p,q, \setR)$ is equivalent to $g\in \SOp(p,q, \setR) \cdot Z_{\O(p,q,\setR)}(x)$. 	
    Since $\SOp(p,q,\setR)$ is the identity component of $\O(p,q, \setR)$, the $\SOp(p,q,\setR)$-cosets are the connected components of $\O(p,q,\setR)$, which proves the second part.
    
\end{proof}

Let us first justify that the question of conjugacy can be treated at the level of the standard triples.
\begin{lemma}\label{lmno:CentralizerTriple}
    Let $(x,y,h)$ be a standard triple in $\so(p,q, \setR)$.
    Let $g\in Z_{O(p,q)}(x)$. Then there exists $g'\in \O(p,q,\setR)$ in the same connected component as $g$ that centralizes the standard triple $(x,y,h)$.
\end{lemma}
\begin{proof}
    Since $g$ acts on $\so(p,q, \setR)$ by an automorphism of Lie algebra, $(\Ad_g x = x, \Ad_g y, \Ad_g h)$ is a standard triple in $\so(p,q, \setR)$.
    According to Theorem~\ref{thmno:KostantConjugacy}, there exists $h\in Z_{\SO^+(p,q, \setR)}(x)$ such that $hg$ centralizes the standard triple $(x,y,h)$. Since $h$ is in the identity component of $\O(p,q, \setR)$, $hg$ belongs to the same connected component as $g$.
\end{proof}

Tracking the connected components of $\O(p,q,\setR)$ will require some setup.
\begin{definition}
    The group of connected components $\pi_0(\O(p,q,\setR))$ is isomorphic to $\setZ_2\times\setZ_2$ (we assume $p, q \geqslant 1$).
    We define the group homomorphisms $\sigma, \tau : \O(p,q,\setR)\to \{\pm 1\}$ so that:
    \begin{enumerate}
        \item $\sigma(g)=-1$ when $g$ reverses the orientation of a maximal positive subspace of $\setR^{p+q}$,
        \item $\tau(g)=-1$ when $g$ reverses the orientation of a maximal negative subspace of $\setR^{p+q}$.
    \end{enumerate}

    We extend this definition in the case $p=0$ (resp. $q=0$), where $\sigma=1$ (resp. $\tau=1$) and $\tau=\det$ (resp. $\sigma = \det$).
\end{definition}
\begin{remark}
    We use the notation $\sigma, \tau$ indifferently of the orthogonal group on which we use them. When needed, we will explicitly mention the group.
\end{remark}

\begin{proposition}\label{propno:tausigma}
    Let $g_1\in \O(p_1,q_1,\setR)$ and $g_2\in \O(p_2, q_2, \setR)$.
    \begin{enumerate}
        \item Consider $g_1 \oplus g_2\in \O \left( \setR^{p_1+q_1} \oplus \setR^{p_2+q_2} \right) \simeq \O(p_1+p_2, q_1 + q_2, \setR)$. Then
        \begin{align*}
            \tau(g_1 \oplus g_2) &= \tau(g_1)\tau(g_2)\\
            \sigma(g_1 \oplus g_2) &= \sigma(g_1)\sigma(g_2)
        .
        \end{align*}
        \item Consider $\id \otimes g\in \O \left( \setR^{p_1+q_1} \otimes \setR^{p_2+q_2} \right) \simeq \O(p_1p_2 + q_1q_2, p_1q_2 + q_1p_2, \setR)$. Then
        \begin{align*}
            \tau(\id \otimes g) &= \tau(g)^{p_1} \sigma(g)^{q_1}\\
            \sigma(\id \otimes g) &= \sigma(g)^{p_1} \tau(g)^{q_1}
        .
        \end{align*}
    \end{enumerate}
\end{proposition}

\begin{proof}
    ~\begin{enumerate}
        \item Consider the decomposition of $\setR^{p_1+q_1} \oplus \setR^{p_2+q_2}$ into maximal positive/negative subspaces 
        \[
            \setR^{p_1+q_1}\oplus \setR^{p_2+q_2}
            = \left( \setR^{p_1} \oplus \setR^{p_2} \right)
            \oplus 
            \left( \setR^{q_1} \oplus \setR^{q_2} \right)
        .\]
        Assume that for $i=1,2$ $g_i\in \O(p_i, q_i, \mathbb R)$ preserves the decomposition $\setR^{p_i} \oplus \setR^{q_i}$. This can be done without loss of generality since such an element exists in every connected components and the values of $\tau, \sigma$ are left unchanged inside every such component.
        
        Recalling that $\Lambda^n(\mathbb{R}^n)$ has dimension $1$ for all $n\in \setN$ along with the isomorphism
        \[
            \Lambda^{p_1+p_2} \left( \setR^{p_1}\oplus \setR^{p_2} \right)
                \simeq
            \Lambda^{p_1} \setR^{p_1} \otimes \Lambda^{p_2} \setR^{p_2}
        ,\]
        one sees that $g_1\oplus g_2$ preserves the subspace $\setR^{p_1} \oplus \setR^{p_2}$ and multiply its orientation by a factor $\sigma(g)=\sigma(g_1)\sigma(g_2)$. The same observation for the negative spaces gives the result.
        
        \item 
        The tensor product $\setR^{p_1+q_1} \otimes \setR^{p_2+q_2}$ decomposes as an orthogonal direct sum 
        \[
            \left( \setR^{p_1+q_1}, \eta_{p_1, q_1} \right)^{\oplus p_2} 
            \oplus_\perp
            \left( \setR^{p_1+q_1}, -\eta_{p_1, q_1} \right)^{\oplus q_2} 
        \]
        with $\eta_{p_1, q_1}$ the inner product of signature $(p_1, q_1)$ on $\setR^{p_1 + q_1}$.
        In the space $\left( \setR^{p_1+q_1}, -\eta_{p_1, q_1} = \eta_{q_1, p_1} \right)$, the notions of positive and negative subspaces are inverted, so that $\tau_{q_1,p_1} = \sigma_{p_1, q_1}$ and $\sigma_{q_1,p_1} = \tau_{p_1, q_1}$. We can now compute, according to 1.,
        \[
            \tau(\id \otimes g)
            = \tau_{p_2, q_2}(g)^{p_1} \tau_{q_2, p_2}(g)^{q_1}
            = \tau(g)^{p_1} \sigma(g)^{q_1}
        \]
        and
        \[
            \sigma(\id \otimes g) 
            = \sigma_{p_2, q_2}(g)^{p_1} \sigma_{q_2, p_2}(g)^{q_1}
            = \sigma(g)^{p_1} \tau(g)^{q_1}
        .\]
    \end{enumerate}
\end{proof}

We need to understand how many connected components of $\O(p,q,\setR)$ intersect $Z_{\O(p,q,\setR)}(x)$.

\begin{lemma}\label{lmno:SigmaTauStabx}
    Let $x$ be a nonzero nilpotent element in $\so(p,q, \setR)$ and $(n_r, p_r, q_r)$ the respective dimensions and signatures of $W_r$ and $\phi_r$.
    \begin{enumerate}
        \item If $\forall m\in \setN$, $p_{2m}=q_{2m}=0$ then $Z_{\O(p,q,\setR)}(x) \subset \SOp(p,q, \setR)$.
        
        Assume in the remaining cases that there exists $m\in \setN$ such that $(p_{2m}, q_{2m}) \neq (0,0)$.
        \item If $\forall k \in \setN, q_{4k} = p_{4k+2} = 0$
        then $Z_{\O(p,q,\setR)}(x)$ meets exactly two connected components of $Z_{\O(p,q,\setR)}(x)$ and $(\sigma, \tau) \big( Z_{\O(p,q,\setR)}(x) \big) = \{\pm1\}\times \{1\}$.
        \item If $\forall k \in \setN, p_{4k} = q_{4k+2} = 0$
        then $Z_{\O(p,q,\setR)}(x)$ meets exactly two connected components of $Z_{\O(p,q,\setR)}(x)$ and $(\sigma, \tau) \big( Z_{\O(p,q,\setR)}(x) \big) = \{1\}\times \{\pm1\} $.
        \item Otherwise, $Z_{\O(p,q,\setR)}(x)$ meets all four connected components of $\O(p,q,\setR)$, namely, $(\sigma, \tau) \big( Z_{\O(p,q,\setR)}(x) \big) = \{\pm1\} \times \{\pm1\}$.  
    \end{enumerate}
\end{lemma}

\begin{proof}
    Let $(x,y,h)$ be a standard triple in $\so(p,q, \setR)$. According to Lemma~\ref{lmno:CentralizerTriple}, it is equivalent to study the values of $\tau$ and $\sigma$ on the centralizer $Z_{\O(p,q,\setR)}(<x,y,h>)$ or on $Z_{\O(p,q,\setR)}(x)$.
    According to Proposition~\ref{propno:sl2moduleIsometry}, there is an isomorphism
    \[
    	Z_{\O(p,q,\setR)}(<x,y,h>)
    	\simeq 
    	\prod_r 
    		\O(W_r, \phi_r)	
    \]
    where $\O(W_r, \phi_r)$ denotes the automorphisms of $W_r$ that preserve the bilinear form $\phi_r$. Note that for odd $r$, $\phi_r$ is not symmetric, but is a symplectic form. Hence $\O(W_r, \phi_r)$ is a symplectic group and is connected.
    
    First, if $W_r = 0$ for all even $r$, then $Z_{\O(p,q,\setR)}(<x,y,h>)$ is a product of symplectic groups. It is therefore connected, thus contained in $\SOp(p,q, \setR)$.
    This proves moreover that in the general case, only the groups $\O(W_r, \phi_r)$ for even $r$ matter.
    
    We now look at the general case.
    Let $r=2m$ be an even integer with $(p_r, q_r)\neq (0,0)$. Then $B_r$ has signature $(m+1, m)$ and $\phi_r$ has signature $(p_r, q_r)$. For any $g\in \O(W_r, \phi_r)$, we compute according to Proposition~\ref{propno:tausigma}:
    \[\begin{aligned}
    	\sigma(\id_{V_{2m}} \otimes g) &= \sigma(g)^{m+1} \tau(g)^{m}\\
    	\tau(\id_{V_{2m}} \otimes g) &= \tau(g)^{m+1} \sigma(g)^{m}
    \end{aligned}
    \quad.\]
    In consequence, 
    \[
    (\sigma, \tau) \big( \id_{V_{2m}} \otimes \O( W_{2m}, \phi_{2m}) \big)
    =
    \sigma \big( \id_{V_{2m}} \otimes \O( W_{2m}, \phi_{2m}) \big)
    \times
    \tau \big( \id_{V_{2m}} \otimes \O( W_{2m}, \phi_{2m}) \big)
    \]
    and the images of $\sigma$ and $\tau$ are
    \[
    	\sigma \left( \id_{V_{2m}} \otimes \O( W_{2m}, \phi_{2m}) \right)
    		= \begin{cases}
    			\{\pm 1 \}
    			\text{ if $m$ is even and $p_{2m} \geqslant 1$ or $m$ is odd and $q_{2m} \geqslant 1$,}\\
    			\{ 1 \}
    			\text{ otherwise}    			
    		\end{cases}
    \]
    and
    \[
    	\tau \left( \id_{V_{2m}} \otimes \O( W_{2m}, \phi_{2m}) \right)
    		= \begin{cases}
    			\{\pm 1 \}
    			\text{ if $m$ is even and $q_{2m} \geqslant 1$ or $m$ is odd and $p_{2m} \geqslant 1$,}\\
    			\{ 1 \}
    			\text{ otherwise.}    			
    		\end{cases}
    \]
    
    Finally, Proposition~\ref{propno:tausigma} implies that
    \begin{align*}
    	\tau\left( Z_{\O(p,q,\setR)}(<x,y,h>) \right)
    	&= \prod_r \tau \left( \id_{V_{r}} \otimes \O( W_r, \phi_r) \right)\\
    	\sigma\left( Z_{\O(p,q,\setR)}(<x,y,h>) \right)
    	&= \prod_r \sigma \left( \id_{V_{r}} \otimes \O( W_r, \phi_r) \right)
    \end{align*}
    where $\prod_r$ is not a direct product of set but denotes the element-wise product of subsets of $\{\pm 1\}$. We conclude that
    \begin{itemize}
	    \item 
	    $\sigma \left( Z_{\O(p,q,\setR)}(<x,y,h>) \right) = \{\pm1\}$ if and only there exists an even $m$ with $p_{2m} \geqslant 1$ or an odd $m$ with $q_{2m} \geqslant 1$,
	    \item 
	    $\tau \left( Z_{\O(p,q,\setR)}(<x,y,h>) \right) = \{\pm1\}$ if and only there exists an even $m$ with $q_{2m} \geqslant 1$ or an odd $m$ with $p_{2m} \geqslant 1$.
	\end{itemize}   
\end{proof}

\begin{lemma}\label{lmno:SigmaTauMinusx}
~\begin{enumerate}
	\item 
	Let $r=2m$ and write
	\[
		V_r = \bigoplus_{i=0}^{r} V_{r}(r-2i)
	.\]
    the decomposition of $V_r$ in weight spaces.
	Define
	\[
		T := \tilde u_r= \bigoplus_i (-1)^{m-i} \id_{V_{r}(r-2i)}
	\quad .\]
	Then in $\O(B_r)\simeq \O(m+1, m, \setR)$ it holds  $\sigma(T)=\tau(T) = (-1)^{m/2}$ when $m$ is even and  $\sigma(T) =\tau(T) = (-1)^{(m+1)/2}$ when $m$ is odd.
%
	\item Let $r,r'$ be odd integers and consider $V_r$ and $V_{r'}$, equipped with their respective invariant symplectic forms $B_r$ and $B_{r'}$.
	They each decompose into weight spaces and we define
	\begin{align*}
		T:=u_r &= \bigoplus_i (-1)^i \id_{V_{r}(r-2i)} \in \End(V_r)\\
		T':=u_{r'} &= \bigoplus_{i'} (-1)^{i'} \id_{V_{r'}(r-2i')} \in \End(V_{r'})	
	.\end{align*}
	
	Equipping $V\otimes V'$ with the split inner product $B_r\otimes B_{r'}$, 
	\[\begin{aligned}
		\sigma(T\otimes T')
		 = \tau(T\otimes T')
		 &= (-1)^{(r+1)(r'+1)/4}
	\end{aligned}
	.\]
\end{enumerate}
\end{lemma}
\begin{proof}
	~\begin{enumerate}
	\item For every $k\in 2\llbracket -m, m \rrbracket$, let $e_{k}$ be a vector spanning $V_{r}(k)$, with the furthermore requirement that $B(e_k, e_{-k})=1$. Then 
	\begin{align*}
		\forall i \in \llbracket 0, m - 1 \rrbracket, \;\;
			& T(e_{2m-2i} \pm e_{2i-2m}) = (-1)^{m-i} (e_{2m-2i} \pm e_{2i-2m})
	\end{align*}
	and $T(e_0)=e_0$.
	Furthermore, 
	\begin{itemize}
	\item 	$\Vect(e_0, e_2 + e_{-2}, \: \dots \: , e_{2m} + e_{-2m})$
			 is a maximal positive subspace,
	\item	$\Vect(e_2 - e_{-2}, \: \dots \: , e_{2m} - e_{-2m})$
			 is a maximal negative subspace.
	\end{itemize}
	
	Counting signs yields (note that $\det(T)=1$):
	when $m$ is even, 
	\(
		\sigma(T) = \tau(T) = (-1)^{m/2}
	\) 
    and
	when $m$ is odd, 
	\(
		\sigma(T) = \tau(T) = (-1)^{(m+1)/{2}}
	\).
    
	\item Write $r= 2m+1$ and $r' = 2m'+1$. The set of odd relative integers between $-r$ and $r$ is $1 + 2\llbracket -(m+1), m \rrbracket$. For $k \in 1 + 2\llbracket -(m+1), m \rrbracket$, let $e_k \in V_r(k)$ such that $B_r(e_k, e_{-k}) = 1$ when $k \geqslant 0$. Similarly, for $k' \in 1 + 2\llbracket -(m'+1), m' \rrbracket$, let $f_{k'} \in V_{r'}(k')$ such that $B_{r'}(f_{k'}, f_{-k'})=1$ when $k'\geqslant 0$.
	A basis of $V_r\otimes V_{r'}$ orthogonal for $B_r\otimes B_{r'}$ is given by the following family:
	\begin{equation}\label{eqno:VV'Basis}
	\begin{multlined}
		\Big(
			 v_k\otimes f_{k'} + v_{-k} \otimes f_{-k'}, \quad
			 v_k\otimes f_{k'} - v_{-k} \otimes f_{-k'}, \;\\
			 v_k\otimes f_{-k'} + v_{-k} \otimes f_{k'}, \quad 
			 v_k\otimes f_{-k'} - v_{-k} \otimes f_{k'}				
		\Big)_{
		\substack{
			k \in 1 + 2\llbracket 0, m \rrbracket\\
			k'\in 1 + 2\llbracket 0, m' \rrbracket}
		}
	\end{multlined}\end{equation}
	More precisely, for every
	$k \in 1 + 2\llbracket 0, m \rrbracket, \, 
	k' \in 1 + 2\llbracket 0, m' \rrbracket$,
    the basis vectors are normed as follows:
	\begin{equation*}
	\begin{aligned}
		B_r\otimes B_{r'}(
		v_k\otimes f_{k'} \pm v_{-k} \otimes f_{-k'}, \quad
		v_k\otimes f_{k'} \pm v_{-k} \otimes f_{-k'}
		)
		= \pm 2\\
		B_r\otimes B_{r'}(
		v_k\otimes f_{-k'} \pm v_{-k} \otimes f_{k'}, \quad
		v_k\otimes f_{-k'} \pm v_{-k} \otimes f_{k'}
		)
		= \mp 2
	\end{aligned}
	\end{equation*}
	with a single, coherent choice of sign $\pm$ for each line.
	Let us define 
	\[
		V_{\pm_1 \pm_2} := \Vect(v_k\otimes f_{\pm_1 k'} \, \pm_2 v_{-k} \otimes f_{\mp_1 k'}
		)_{\substack{k \in 1 + 2\llbracket 0, m \rrbracket\\
					k'\in 1 + 2\llbracket 0, m' \rrbracket}}	
	\]
	so that they are pairwise orthogonal, all have dimension $(m+1)(m'+1)$, $V_{++}$ and $V_{--}$ are positive subspaces and $V_{+-}$ and $V_{-+}$ are negative subspaces.
	
	The action of $T\otimes T'$ is diagonal in the basis~\eqref{eqno:VV'Basis} and takes the following form (recall that $k$ and $k'$ are odd):
	\[
	T\otimes T': \quad
	\begin{cases}
				 v_k\otimes f_{k'} \pm v_{-k} \otimes f_{-k'}
				 &\mapsto (-1)^{(r+r'+k+k')/2}\left(
				 	v_k\otimes f_{k'} \pm v_{-k} \otimes f_{-k'} \right)
				 \\
				 v_k\otimes f_{-k'} \pm v_{-k} \otimes f_{k'}				
				 &\mapsto (-1)^{(r+r'+k-k')/2}\left(
				 	v_k\otimes f_{-k'} \pm v_{-k} \otimes f_{k'} \right)
	\end{cases}
	\quad. \]	
	\end{enumerate}
	
	Counting the signs, we obtain:
	\begin{itemize}
	\item If either $m$ or $m'$ is odd, in each of the four subspaces $V_{\pm\pm}$, the eigenvalues $1$ and $-1$ have the same multiplicity and
	\[\begin{aligned}
		\det\nolimits_{V_{\pm_1\pm_2}}(T\otimes T') &= (-1)^{(m+1)(m'+1)/2}					
	\end{aligned}\]
	thus
	\[
		\sigma(T\otimes T') = \tau(T\otimes T') = 1 = (-1)^{(m+1)(m'+1)}
	.\]
	
	\item If both $m$ and $m'$ are even then
	\[\begin{aligned}
		\det\nolimits_{V_{+ \pm}}(T\otimes T')
		&= (-1)^{1+(r+r')/2}(-1)^{((m+1)(m'+1)-1)/2} \\					
		\det\nolimits_{V_{- \pm}}(T\otimes T')
		&= (-1)^{(r+r')/2}(-1)^{((m+1)(m'+1)-1)/2} \\					
	\end{aligned}\]
	and
	\[
		\sigma(T\otimes T') = \tau(T \otimes T') = -1 = (-1)^{(m+1)(m'+1)}
	.\]	
	\end{itemize}
\end{proof}

We now have all the results needed to know when $x$ and $-x$ belong to the same adjoint orbit.
\begin{theorem}
	Let $(x,y,h)$ be a standard triple in $\so(p,q, \setR)$.
	For $r\geqslant 1$, define $n_r = \dim W_r$. 
	Then
    $x$ and $-x$ belong to different $\SOp(p,q,\setR)$-orbits if and only if     
        \[
        \sum_{k\geqslant 0} n_{8k+4} + n_{8k+2} + \frac{n_{4k+1}}2 
        \equiv 1
        \mod 2
        \]
    and either
        \begin{equation} \label{eqno:SOpqCond1}
            q_{4k} = p_{4k+2} = 0 \qquad  \forall k \in \setN
        \end{equation}
		or
        \begin{equation}\label{eqno:SOpqCond2}
             p_{4k} = q_{4k+2} = 0\qquad \forall k \in \setN
        .\end{equation}

    When $x$ and $-x$ are $\SOp(p,q,\setR)$-conjugate, the nilpotent orbit of $x$ admits no Gibbs states.
\end{theorem}

\begin{proof}
	Let us construct an element of $\O(p,q,\setR)$ that conjugates $x$ and $-x$. Let 
	\[
		V = \bigoplus_r V_r \otimes W_r
	\]
	be the isotypic decomposition associated with the standard triple. 
	For even $r$, define $A_r = T_r\otimes \id_{W_r}$ with $T_r$ as defined in Lemma~\ref{lmno:SigmaTauMinusx}. 
	For odd $r$, $W_r$ is endowed with a symplectic product $\phi_i$ and it is symplectically isometric to $V_{n_r - 1}$ (disregarding the action of $\sl(2, \setR)$). We thus define $T'_r \in \O(W_r, \phi_r)$ that is symplectically equivalent to $T_{n_r - 1}$ on $V_{n_r - 1}$, and, in particular, $(T'_r)^*\phi_r = -\phi_r$. Therefore $(T_r\otimes T'_r)^* B_r\otimes \phi_r = B_r \otimes \phi_r$ and we define $A_r = T_r\otimes T'_r$. 
	
	Finally, we gather these maps into
	\[
		A = \bigoplus_r A_r
        \in \O \left( \bigoplus V_r\otimes W_r \right)
		.
	\]
	By construction, $A\in \O(p,q,\setR)$ and $AxA^{-1} = -x$ (see Proposition~\ref{propno:x->-x}).
	We simply need to compute the connected component of $\O(p,q,\setR)$ to which $A$ belongs, using Lemma~\ref{lmno:SigmaTauMinusx}, and compare it to the connected components that intersect the centralizer of $x$ in $\O(p,q,\setR)$, given by Lemma~\ref{lmno:SigmaTauStabx}.

	First, when neither condition~\eqref{eqno:SOpqCond1} or~\eqref{eqno:SOpqCond2} holds, $\SOp(p,q, \setR)$ acts transitively on the $\O(p,q,\setR)$-orbit according to Lemma~\ref{lmno:SigmaTauStabx}, so that $x$ and $-x$ are $\SOp(p,q,\setR)$-conjugate. 
    
	We now assume that at least one of the conditions~(\ref{eqno:SOpqCond1},\ref{eqno:SOpqCond2}) holds. Since all $A_r$ have determinant $1$, $\det(A) =1$. Lemma~\ref{lmno:SigmaTauStabx} implies that $Z_{\SO(p,q)}(x) = Z_{\SOp(p,q, \setR)}(x)$, thus $x$ is $\SOp(p,q, \setR)$-conjugate to $-x$ if and only if $A \in \SOp(p,q, \setR)$. It is therefore a matter of computing $\sigma(A)$. We treat separately the various remainders of $r$ modulo $4$:
	\begin{itemize}
	\item $r=4k$: we know that $\sigma(T_r) = \tau(T_r) = (-1)^k$ therefore 
	$\sigma(A_r) = (-1)^{kn_{4k}}$.

	\item $r=4k+2$: we know that $\sigma(T_r) = \tau(T_r) = (-1)^{k+1}$ therefore 
	$\sigma(A_r) = (-1)^{(k+1)n_{4k+2}}$.

	\item $r=2m+1$: we know that $\sigma(A_r)=\sigma(T_r\otimes T_{n_r -1}) = (-1)^{(m+1)n_{2m+1}/2}$.
	\end{itemize}
	
	We conclude that $A\in \SOp(p,q, \setR)$ if and only $\sigma(A)=1$, namely 
	\[\begin{aligned}
		0&
        \equiv
        \sum_{k\geqslant 0}
			 k n_{4k}
			 + (k+1)n_{4k+2}
		+ \sum_{m\geqslant 0} (m+1) \frac{n_{2m+1}}2
		\mod 2
        \\
        &\equiv
        \sum_{k'\geqslant 0}
			 (2k'+1) n_{8k'+4}
        + \sum_{k'\geqslant 0}
			 (2k'+1)n_{8k'+2}
		+ \sum_{m'\geqslant 0} (2m'+1) \frac{n_{4m'+1}}2
		\mod 2
        \\
		&\equiv 
        \sum_{k\geqslant 0} n_{8k+4} + n_{8k+2} + \frac{n_{4k+1}}2 
		\mod 2
	.
    \end{aligned}\]
\end{proof}

The results are summed up in Table~\ref{tableno:Conditions}.

\begin{table}[h!]
\caption{Conditions for $\Orb = - \Orb$.}
\[\begin{array}{|c|c|}
	\hline
	\sl(n, \setR)
	 & 
	 \exists r \not\equiv 1 \mod 4, \, \; n_r \neq 0
	\\	\hline
	\so(p,q, \setR)
	&
	\begin{array}{c}
		\big[\left(
			\exists k\in \setN, q_{4k} \neq 0 \text{ or } p_{4k+2}, \neq 0
		\right)\\
		\text{ and }
		\left(
			\exists k\in \setN, p_{4k} \neq 0 \text{ or } q_{4k+2}, \neq 0
		\right)
		\big]
		\text{ or }\\
		 \sum_{k \geqslant 0} n_{8k+2} + n_{8k+4} + n_{4k+1}/2 \equiv 0 \mod 2
	\end{array}
  	\rule{0pt}{5.6ex} 
	\\	\hline
	\sl(n, \setH)
	&
	\textit{always}
	\\	\hline
	\symp(p,q)
	&
	\textit{always}
	\\	\hline
	\symp(2n, \setR), 	\su(p,q), \so^*(2n)
	&
	\forall m, \; p_{2m+1} = q_{2m+1}
	\\	\hline
	\text{complex Lie algebras}
	&
	\textit{always}
	\\ \hline
\end{array}\]
\label{tableno:Conditions}
\end{table}

\section{Conclusion}

We have proved that in many cases, nilpotent adjoint orbits of simple real Lie groups are stable under negation. 
For complex Lie algebra, symplectic Lie algebras and quaternionic special linear algebra, the orbits are always stable under negation.
For the real symplectic, unitary and quaternion unitary groups, the orbits are stable under negation if and only if every multiplicity space of odd weight has split signature.
For real special linear algebras and orthogonal Lie algebras, the criterion has to do with the existence of certain multiplicity spaces, possibly with appropriate signature.
\printbibliography

\end{document}